\documentclass[10pt,]{article}
\usepackage{xspace}
\usepackage{amssymb}
\usepackage{amsmath,amscd}
\usepackage{amsfonts}
\usepackage{amsthm}
\usepackage{keyval}
\usepackage{array}
\usepackage{pictexwd,dcpic}
\usepackage[all,cmtip]{xy}
\usepackage{setspace}
\newtheorem{prop}{Proposition}[section]

\newtheorem{theorem}{Theorem}[section]

\newtheorem{cor}{Corollary}[section]

\begin{document}

\pagenumbering{roman}

\thispagestyle{empty}

\title{On the $ER(2)$-cohomology of some odd-dimesional projective spaces}
\author{Romie Banerjee}
\date{}
\maketitle

\begin{abstract}
 Kitchloo and Wilson have used the homotopy fixed points spectrum $ER(2)$ of the classical complex-oriented Johnson-Wilson spectrum $E(2)$ to deduce certain non-immmersion results for real projective spaces. $ER(n)$ is a $2^{n+2}(2^n-1)$-periodic spectrum. The key result to use is the existence of a stable cofibration $\Sigma^{\lambda(n)}ER(n) \rightarrow ER(n) \rightarrow E(n)$ connecting the real Johnson-Wilson spectrum with the classical one. The value of $\lambda(n)$ is $2^{2n+1}-2^{n+2}+1$. We extend Kitchloo-Wilson's results on non-immersions of real projective spaces by computing the second real Johnson-Wilson cohomology $ER(2)$ of the odd-dimensional real projective spaces $RP^{16K+9}$. This enables us to solve certain non-immersion problems of projective spaces using obstructions in $ER(2)$-cohomology. 

\vspace{3mm}

\end{abstract}

\section{Introduction}

The spectrum $MU$ of complex cobordism comes naturally equipped with an action of $\mathbb{Z}/2$ by complex conjugation. Hu and Kriz in \cite{HK} have used this action to construct genuine $\mathbb{Z}/2$ equivariant spectra $E\mathbb{R}(n)$ from the complex-oriented spectra $E(n)$. Kitchloo and Wilson in \cite{KW1} have used the homotopy fixed point spectrum of this to solve certain non-immersion problems of real projective spaces. The homotopy fixed point spectrum $ER(n)$ is $2^{n+2}(2^n-1)$-periodic compared to the $2(2^n-1)$-periodic $E(n)$. The spectrum $ER(1)$ is $KO_{(2)}$ and $E(1)$ is $KU_{(2)}$.

Kitchloo and Wilson have demonstrated the existence of a stable cofibration connecting $E(n)$ and $ER(n)$,
\begin{equation}
\begin{CD}
\Sigma^{\lambda(n)} ER(n) @>x>> ER(n) @>>> E(n)
\end{CD}
\end{equation}
where $\lambda(n) = 2^{2n+1}-2^{n+2}+1$.
This leads to a Bockstein spectral sequence for $x$-torsion. It is known that $x^{2^{n+1}-1}=0$ so there can be only $2^{n+1}-1$ differentials. For the case of our interest $n=2$ there are only 7 differentials.

From \cite{James} we know that if there is an immersion of $RP^b$ to $\mathbb{R}^c$ then there is an axial map
\begin{equation}
RP^b \times RP^{2^L-c-2} \rightarrow RP^{2^L-b-2}.
\end{equation}
For $b=2n$ and $c=2k$ Don Davis shows in \cite{Davis84} that there is no such map when $n=m+\alpha(m)-1$ and $k=2m-\alpha(m)$, where $\alpha(m)$ is the number of ones in the binary expression of $m$ by finding an obstruction to James's map (2) in $E(2)$-cohomology. Kitchloo and Wilson get new non-immersion results by computing obstructions in $ER(2)$-cohomology. In this paper we extend Kitchloo-Wilson's results by computing the $ER(2)$-cohomology of the odd projective space $RP^{16K+9}$. This will give us newer non-immersion results. The main results are the following.

\begin{theorem}
A 2-adic basis of $ER(2)^{8*}(RP^{16K+9},*)$ is given by the elements
$$\alpha^ku^j, \,\, (k \geq 0,1\le j \le 8K+4)$$
$$v_2^4\alpha^k u^j, \,\, (k \geq 1, 1\leq j \leq 8K+4)$$
$$v_2^4u^j,\,\,(4 \leq j \leq 8K+4)$$
$$x \alpha^ki_{16K+9}, \,\, xv_2^4\alpha^k i_{16K+9},\,\,(k \geq 0)$$
\end{theorem}

\begin{theorem}
Let $\alpha(m)$ be the number of ones in the binary expansion of $m$. If $(m,\alpha(m)) \equiv$ (6,2) or (1,0) mod 8,

$RP^{2(m+\alpha(m)-1)}$ does not immerse in $\mathbb{R}^{2(2m -\alpha(m))+1}$.
\end{theorem}

This shall give us new non-immersions that are often new and different from those of \cite{KW1} and \cite{KW2}. Using Davis's table \cite{Davis} the first new result is $RP^{2^{13}-2}$ does not immerse in $\mathbb{R}^{2^{14}-59}$.

\vspace{3mm}
{\bf Acknowledgements}\, This paper came out of my PhD dissertation at Johns Hopkins. I owe many thanks to my advisor Steve Wilson and Michael Boardman.
\tableofcontents

\pagenumbering{arabic}

\section{The Bockstein spectral sequence}

The results obtained in this section can be found in \cite{KW1}. We reproduce it here for the convenience of the reader.

We have the stable cofibration 
$$\xymatrix{
\Sigma^{\lambda(n)}ER(n) \ar[r]^x &ER(n) \ar[r] &E(n)\\
}$$
where $x\in ER(n)^{-\lambda(n)}$ and $\lambda(n) = 2^{2n+1}-2^{n+2}+1$. The fibration gives us a long exact sequence 
\begin{equation}
\xymatrix{
ER(n)^*(X)  \ar[rr]^x &&ER(n)^*(X) \ar[dl]^{\rho}\\
&E(n)^*(X) \ar[ul]^{\partial}\\
}
\end{equation}
where $x$ lowers the degree by $\lambda(n)$ and $\partial$ raises the degree by $\lambda(n)+1$. This leads to the Bockstein spectral sequence, which will completely determine $M=ER(n)^*(X)/(x)$ as a subring of $E(n)^*(X)$. We know that $x^{2^{n+1}-1}=0$ so there can be only $2^{n+1}-1$ differentials. 

We filter $M$,
$$ 0= M_0 \subset M_1 \subset M_2 \subset \ldots \subset M_{2^{n+1}-1} =M$$ by submodules $$M_r = \hbox{Ker}\left[x^r:\frac{ER(n)^*(X)}{x} \rightarrow \frac{x^rER(n)^*(X)}{x^{r+1}}\right]$$ so that $M_r/M_{r-1}$ gives the $x^r$-torsion elements of $ER(n)^*(X)$ that are non-zero in $M$. 

We collect the basic facts about the spectral sequence in the following theorem. $E(n)$ is a complex oriented spectrum with a complex conjugation action. Denote this action by $c$.

\begin{theorem} \cite[Theorem 4.2]{KW1} In the Bockstein spectral sequence for $ER(n)^*(X)$
\begin{enumerate}
\item The exact couple (3) gives rise to a spectral sequence, $E^r$, of $ER(n)^*$-modules, starting with $$E^1\simeq E(n)^*(X).$$
\item $E^{2^{n+1}}=0$
\item $\hbox{Im}\, d^r \simeq M_r/M_{r-1}$. 
\item The degree of $d^r$ is $r\lambda(n)+1$.
\item $d^r(ab) = d^r(a)b +c(a)d^r(b)$
\item $d^1(z) = v_n^{-(2^n-1)}(1-c)(z)$ where $c(v_i)=-v_i$.
\item If $c(z)=z$ in $E^1$, then $d^1(z)=0$. If $c(z)=z$ in $E^r$ then $d^r(z^2)=0$.
\item The following are all vector spaces over $\mathbb{Z}/2$:
$$M_j/M_i,\, (j\geq i> 0)\,\, \hbox{and}\, E^r, (r\geq 2).$$ 
 
\end{enumerate}
\end{theorem}

\vspace{5mm}
Note that the image of $ER(n)^*(X) \rightarrow E(n)^*(X)$ consists of targets of the differentials and therefore always have the differentials trivial on them. Also anything in the image is trivial under the action of $c$.

Since $ER(n)^*(-)$ is $2^{n+2}(2^n-1)$-periodic we will consider it as graded over $\mathbb{Z}/(2^{n+2}(2^n-1))$. We have to do the same then for $E(n)^*(-)$. We can do this by setting the unit $v_n^{2^{n+1}}= 1$ in the homotopy of $E(n)$. (This does not lose any information since we can always recover the original by inserting powers of $v_n^{2^{n+1}}$ to make the degrees match.)

\subsection{The spectral sequence for $ER(2)^*$}

For the Bockstein spectral sequence of Theorem 2.1 to be useful, we need to know the ring $ER(n)^*$. From now on we concentrate on the case $n=2$. This spectral sequence begins with $E^1= E(2)^*$ which is just a free $\mathbb{Z}_{(2)}[v_1]$-module on a basis given by $v_2^i$ for $0\leq i<8$. We are grading mod 48. Since all elements of $E^1$ have even degree and degree $\deg d^r=17r+1$, $d^r=0$ for $r$ even. As $E^8=0$, we only have $d^1,d^3,d^5,d^7$ to consider.

We have the differential $d^1$ acting as follows:
$$d^1(v_2^{2s+1})= v_2^{-3}(1-c)v_2^{2s+1}= v_2^{-3}2v_2^{2s+1}= 2v_2^{2s-2}.$$
Similarly $d^1(v_2^{2s})= 0$. 

However multiplication by $v_1$ doesn't behave well with respect to this differential. The problem is that $v_1 \in E(2)^*$ does not lift to $ER(2)^*$. We will need a substitute for $v_1$. We shall use the element $\alpha = v_1^{ER(2)}  \in ER(2)^{\lambda(2)-1}$ (see \cite[page 13]{KW1}). The image of $\alpha$ is $v_1v_2^5 \in E(2)^{-32}$. Because $v_2$ is a unit this a good substitute for ordinary $v_1$. Furthermore this is invariant under $c$ because it is in the image of the map from $ER(2)^*$ and is a permanent cycle. Or we could just see this by observing that there are an even number of $v$'s. We rewrite the homotopy of $E(2)$ as $\mathbb{Z}_{(2)}[\alpha,v_2^{\pm 1}]$ but again set $v_2^8=1$. 

Now back to the computation of $d^1$ on $v_2^{2s+1}$ where the $E_1$ term is a free $\mathbb{Z}_{(2)}[\alpha]$-module on generators $v_2^i,\, 0\leq i<8$.
$$d^1(\alpha^k v_2^{2s+1}) = d^1(\alpha^k)v_2^{2s+1} + c(\alpha^k)d^1(v_2^{2s+1}) = 0+\alpha^k 2v_2^{2s-2}.$$ But this really follows from the fact that we have a spectral sequence of $ER(2)^*$-modules. 
Thus the $d^1$-cycles form a free $\mathbb{Z}_{(2)}[\alpha]$-module generated by $\{ 1, v_2^2, v_2^4, v_2^6 \}$ and the $d^1$-boundaries form the free submodule with basis $\{ \alpha_0,\alpha_1,\alpha_2,\alpha_3 \}$, where $\alpha_i = 2v_2^{2i}$. In particular, $\alpha_0=2$. Thus $E^2 = E^3$ is the free $\mathbb{F}_2[\alpha]$-module with the basis (the images of) $\{ 1, v_2^2,v_2^4,v_2^6 \}$.

By \cite{KW1}, $d^3(v_2^2) =\alpha v_2^4$. Since $d^3$ is a derivation, $d^3(v_2^6)=\alpha$, and the only elements of $E^4$ are 1 and $v_2^4$. Since $\deg d^5= 38$, $d^5=0$. We must have $d^7(v_2^4)=1$ to make $E^8=0$.

We can read off $M$ as a $\mathbb{Z}_{(2)}[\alpha]$-submodule of $E^1$. $M_1 = \hbox{Im}d^1$ is the free module on basis $\{ \alpha_0,\alpha_1,\alpha_2,\alpha_3 \}$. $M_3$ is generated as a module by adding the elements $\alpha$ and $w = \alpha v_2^4$, which make $M_3/M_1$ the free $\mathbb{F}_2[\alpha]$-module with basis $\{ \alpha, w \}$. Finally, the only new element of $M_7$ is 1. The rest of the module structure is given by

\begin{equation}
2\alpha = \alpha\alpha_0, 2w=\alpha\alpha_2, 2.1 =\alpha_0, \alpha.1 = \alpha
\end{equation} 

Further, $M$ is a subring of $E^1$, generated by $\alpha_1, \alpha_2, \alpha_3, \alpha$ and $w$. The products not already given are
\begin{eqnarray}
\alpha_s\alpha_t &=& 2\alpha_{s+t}  \,\,\,\,(\hbox{taking} \,\,s+t \,\hbox{mod}\, 4)\\
w\alpha_s &=& \alpha\alpha_{s+2} \,\,\,\,(\hbox{taking} \,\,s+2 \,\hbox{mod}\, 4)\\
w^2 &=& \alpha^2
\end{eqnarray}

To obtain $ER(2)^*$, we must unfilter $M= ER(2)^*/(x)$ and add the generator $x$. We lift each generator $\alpha_s$ and $w$ to $ER(2)^*$, keeping the same names, with the relations 
\begin{equation}
\alpha_sx=0,\alpha x^3=0,wx^3=0,x^7=0.
\end{equation}
By sparseness, each $\alpha_s$ and $w$ lifts {\it uniquely} to $ER(2)^*$; further, the module actions (4) and multiplications (5), (6), (7) also lift uniquely and hold in $ER(2)^*$, not merely mod$(x)$.

\begin{prop}\cite[Section 5]{KW1}
$ER(2)^*$ is graded over $\mathbb{Z}/48$. It is generated as a ring by elements, $$x,w,\alpha,\alpha_1,\alpha_2,\alpha_3$$ of degrees $-17,-8,-32,-12,-24$ and $-36$ respectively, with relations and products as listed above.

\end{prop}

\subsection{The Bockstein Spectral Sequence for $ER(2)^*(RP^{\infty})$}

As always, we have the split short exact sequence $$0 \rightarrow ER(2)^*(X,*) \rightarrow ER(2)^*(X) \rightarrow ER(2)^*(*) \rightarrow 0$$ for any space $X$ with basepoint $*$. Now that we know $ER(2)^*(*) = ER(2)^*$, we will concentrate on $ER(2)^*(X,*)$. Nevertheless, we need to use the action of $ER(2)^*$ on $ER(2)^*(X)$ and hence $ER(2)^*(X,*)$; furthermore, this action extends to actions of the Bockstein spectral sequences.

The $E(2)^*$-cohomology of $RP^{\infty}$ can be computed from the Gysin sequence 
$$\xymatrix{
\ldots E(2)^{k-2}(\mathbb{C}P^{\infty}) \ar[r]^{[2](x_2)} &E(2)^k(\mathbb{C}P^{\infty}) \ar[r] &E(2)^k(\widetilde{RP^{\infty}}) \ar[r] &E(2)^{k-1}(\mathbb{C}P^{\infty}) \ldots\\
}$$
where $E(2)^*(\mathbb{C}P^{\infty}) \simeq E(2)^*[[x_2]]$, $x_2 \in E(2)^{2}(\mathbb{C}P^{\infty})$. From above $E(2)^*(RP^{\infty}) \simeq E(2)^*[[x_2]]/([2](x_2))$. Recall that $ER(2)^*(RP^{\infty}) = ER(2)^*[[u]]/([2](u))$ where we have $u \in ER(2)^{1-\lambda(2)}(RP^{\infty})$. We will replace the $x_2$ by the image of $u \in ER(2)^{-16}(RP^{\infty})$ which we also call $u \in E(2)^{-16}(RP^{\infty})$, which is really $v_2^3x_2$. Likewise we replace the usual $v_1 \in E(2)^{-2}$ with $v_2^5v_1 =\alpha \in E(2)^{-32}$ which comes from $\alpha \in ER(2)^{-32}$. The element $w \in ER(2)^{-8}$ maps to $\alpha v_2^4 = v_2v_1 \in E(2)^{-8}$. These changes are necessary because $x_2$ and $v_1$ are not in the image of $ER(2)$-cohomology.

We will describe our groups in terms of a {\it 2-adic basis} in the sense of \cite{KW1}, i.e, a set of elements such that any element in our group can be written as a unique sum of these elements with coefficients 0 or 1 (where the sum is allowed to be a formal power series in $u$). In the ring $E(2)^*(RP^{\infty})$ we have $2u = \alpha u^2 + \ldots$, therefore the 2-adic basis is given by $v_2^i\alpha^ku^j, \,(0\leq i <8, 0\leq k, 1\leq j)$.  

The original relation $[2](u) = 2u+_F v_1u^2 +_F v_2u^4$ for $E\mathbb{R}(2)$ converts to the relation 
\begin{equation}
[2](u)=2u+_F \alpha u^2 +_F u^4
\end{equation}
since $v_1$ is replaced by $v_1(v_2^3)^{-1}=v_1v_2^5=\alpha$ and $v_2$ is replaced by $v_2(v_2^3)^{-3} = v_2^{-8}=1$. Because $2x=0$, $x$ times the relation (9) gives us $0=x(\alpha u^2 +_F u^4)$. Therefore from the point of view of $x^1$-torsion $\alpha u^2$ can be replaced with $u^4+\ldots$. Similarly, if we multiply by $x^3$ and use the relation $x^3\alpha =0$ we end up with $x^3u^4 =0$.

Since $u \in E(2)^*(RP^{\infty})$ is in the image from $ER(2)^*(RP^{\infty})$ our differentials commute with multiplication by $u$ and also commute with multiplication by $\alpha$. The $d^1$ differential creates a relation coming from our relation $0=2u +_F \alpha u^2 +_F u^4$ when $2u$ is set to zero. So in $E^2$, we have $\alpha u^2 \equiv u^4+\ldots$. 
The Bockstein spectral sequence goes like this:

\begin{theorem}\cite[Theorem 8.1]{KW1}

$E^1 =E(2)^*(RP^{\infty},*)$ is represented by $$v_2^i\alpha^ku^j \hspace{3mm} (0\leq i\le 7, \hspace{3mm} 0\leq k, \hspace{3mm} 1\leq j).$$ $$d^1(v_2^{2s-5}\alpha^k u^j)=2v_2^{2s}\alpha^k u^j = v_2^{2s}\alpha ^{k+1}u^{j+1}+ \ldots$$

\vspace{5mm}
$E^2=E^3$ is given by: $$v_2^{2s}\alpha^k u, \hspace{3mm} v_2^{2s}u^j \hspace{3mm} (2\leq j, \hspace{3mm} 0\leq s\le 4, \hspace{3mm} 0 \leq k)$$ $$d^3(v_2^{4s-2}\alpha ^k u) = v_2^{4s}\alpha ^{k+1}u,\,\,\,d^3(v_2^{4s-2}u^j) = v_2^{4s}\alpha u^j = v_2^{4s}u^{j+2}+\ldots$$

\vspace{5mm}
$E^4=E^5=E^6=E^7$ is given by:$$ v_2^4u^{\{1-3\}}, \hspace{3mm} u^{\{1-3\}}$$ $$d^7(v_2^4u^{\{1-3\}})=u^{\{1-3\}}.$$

\vspace{5mm}
The $x^1$-torsion generators are given by:$$\alpha_i\alpha ^ku^j \hspace{3mm} (0\leq i, \hspace{3mm} 0\leq k, \hspace{3mm} 1\leq j)$$ where $\alpha_0=2$.

\vspace{5mm}
The $x^3$-torsion generators are given by:$$\alpha^{k+1}u,\, \alpha^k wu \,\,(k \geq 0),\, u^j\,\,(j \geq 4),\, wu^j\,\,(j \geq 2).$$

\vspace{3mm}
The only $x^7$-torsion generators are $$u^{\{1-3\}}.$$
\end{theorem}

In degrees that are multiples of 8 (denoted $8*$), the description of $ER(2)^*(RP^{\infty},*)$ simplifies enormously. As $x$ is the only generator whose degree is not a multiple of 4, and $x^4$ kills everything except powers of $u$, multiplication by powers of $x$ produces no new elements in degree $8*$. 

\begin{cor}\cite[Theorem 2.1]{KW2}
The homomorphism
$$ER(2)^{8*}(RP^{\infty},*) \rightarrow E(2)^{8*}(RP^{\infty},*)$$ is injective, and is almost surjective -- the only elements not hit are $v_2^4u^j$ for $1 \leq j\leq 3$.
\end{cor}

We shall compute the Bockstein spectral sequence associated to $ER(2)$ for the odd-dimensional projective space $RP^{16K+9}$. This gives us nonimmersion results which are often new and differ from those of \cite{KW2}.
We will have to introduce some new elements for the computations in the next section. The Atiyah-Hirzebruch spectral sequence for $ER(2)^*(RP^2)$ gives elements $x_1$ and $x_2$ in filtration degrees 1 and 2. As a $ER(2)^*$ module $ER(2)^*(RP^2)$ is generated by elements we will call $z_{-16}$ represented by $xx_1$ and $z_2$ represented by $x_2$. The element $z_{-16}=u \in ER(2)^*(RP^2)$.

For the cofibration $S^1 \rightarrow RP^2 \rightarrow S^2$, the long exact sequence 
$$\xymatrix{
ER(2)^*(S^1) \ar[dr]_{\partial}  &&ER(2)^*(RP^2) \ar[ll]_{i^*}\\
&ER(2)^*(S^2,*) \ar[ur]_{\rho^*}\\
}$$
is given by $\partial(i_1) = 2i_2, \rho^*(i_2)=z_2$, and $i^*(u)=xi_1$.

We know that $\Sigma^{2n-2}ER(2)^*(RP^2,*) \simeq ER(2)^*(RP^{2n}/RP^{2n-2},*)$. We have elements $z_{2n-18},z_{2n} \in ER(2)^*(RP^{2n}/RP^{2n-2},*)$. These elements map to $v_2^{5n+3}u^n$ and $v_2^{5n}u^n$ respectively in $E(2)^*(RP^{2n}/RP^{2n-2},*)$ by \cite[section 10]{KW1}.

\section{$ER(2)^*(RP^{16K+9},*)$}

In \cite{KW1} and \cite{KW2} Kitchloo and Wilson considered all even dimensional real projective spaces and the odd dimensions $16K+1$. Our object is to extend the results to more odd-dimensional cases.

\begin{prop}
The element $u^{8K+5} \in ER(2)^*(RP^{16K+10})$ maps to a non-zero element of $ER(2)^*(RP^{16K+9})$.
\end{prop}
Note that this cannot happen for a complex oriented cohomology theory.

\vspace{5mm}
\begin{proof} We use the exact sequence 
$$\xymatrix{
ER(2)^*(S^{16K+10},*) \ar[r]^{q^*} &ER(2)^*(RP^{16K+10}) \ar[r] &ER(2)^*(RP^{16K+9})\\
}$$
We only have to show that $u^{8K+5}$ is not in the image of $q^*$. Now $ER(2)^*(S^{16K+10},*)$ is the free $ER(2)^*$-module generated by the element $i_{16K+10}$, and $q^*$ is known.

The structure of $ER(2)^*(RP^{16K+10})$ is given by \cite[Theorem 13.2]{KW1} and \cite[Theorem 13.3]{KW1}. We give a complete description of $M$, where 
$$ER(2)^*(RP^{16K+10},*)/xER(2)^*(RP^{16K+10},*) \simeq M \subset E(2)^*(RP^{16K+10},*)$$ 
as a submodule of $E(2)^*(RP^{16K+10})=\mathbb{Z}_{(2)}[\alpha,v_2^{\pm 1}][u]/(u^{8K+6},[2](u))$. We describe $M$ by specifying a 2-adic basis. As $d^r=0$ for $r$ even, $M$ is filtered by $0=M_0\subset M_1 =M_2 \subset M_3 = M_4 \subset M_5 = M_6 \subset M_7 = M,$ where $M_r/M_{r-1} \simeq \hbox{Im}d^r$. Elements of $M_r$ not in $M_{r-1}$ lift to $x^r$-torsion elements of $ER(2)^*(RP^{16K+10},*)$.

As both $\alpha$ and $u$ come from $ER(2)$-cohomology, we may describe $M$ from section 2 as a filtered $\mathbb{Z}_2[\alpha,u]$-module. We write $z_t$ for various elements of $ER(2)^*(RP^{16K+10},*)$ and $\bar{z_t}$ for its image in $M$, where $t$ denotes the degree.

$\alpha^k u^j(u\alpha_s) \in M_1$ for $k\geq0$, $0\leq j \leq 8K+4$, $0\leq s \leq 3$, where $u\alpha_s = 2v_2^{2s}u = d^1(v_2^{2s+3}u)$. Note that $\alpha_0=2$.

$\alpha^k(u\alpha)\in M_3$ for $k \geq 0$, where $u\alpha = d^3(v_2^{-2}u)$. 

$\alpha^k(uw) \in M_3$ for $k \geq 0$, where $uw = d^3(v_2^2u)= uv_2^4\alpha$.

$\alpha^ku^j\beta_0 \in M_3$ for $k\geq 0$, $0\leq j \leq 8K+1$, where $\beta_0= d^3(v_2^{-2}u^2)=\alpha u^2 \equiv u^4 + \ldots$ mod 2= $\alpha_0$.

$\alpha^k u^j \beta_1 \in M_3$ for $k \geq 0$, $0\leq j \leq 8K+1$ where $\beta_1 = d^3(v_2^2u^2)= w u^2$.

$\alpha^k \gamma_0 \in M_3$ for $k \geq 0$, where $\gamma_0 = v_2\alpha u^{8K+5} = d^3(v_2^{-1} u^{8K+5})$.

$\alpha^k \gamma_1 \in M_3$ for $k \geq 0$, where $\gamma_1 = v_2^5 \alpha u^{8K+5} = d^3(v_2^3 u^{8K+5})$.

$\bar{z}_{16K+10} \in M_5$, where $z_{16K+10} = q^*i_{16K+10}$ is induced from $S^{16K+10}$. Then $\bar{z}_{16K+10} = v_2u^{8K+5} = d^5(v_2^2u ^{8K+4})$. Note that $\alpha \bar{z}_{16K+10} = \gamma_0$ and $w \bar{z}_{16K+10} = \gamma_1$. 

$\bar{z}_{16K-14} \in M_5$, where $\bar{z}_{16K-14} = v_2^5u^{8K+5} = d^5(v_2^6u^{8k+4})$. Note that $\alpha \bar{z}_{16K-14}= \gamma_1$ and $w\bar{z}_{16K-14}=\gamma_0$.

$\bar{z}_{16K+4} \in M_7$, where $\bar{z}_{16K+4} = v_2^2u^{8K+5} = d^7(v_2^6 u^{8K+5})$.

$u^j \in M_7$, for $1 \leq j \leq 3$, since $d^7(v_2^4u^j)=u^j$.

The action of $\alpha$ on $M$ is clear, except for $\alpha \bar{z}_{16K+4} = -u^{8K+4} \alpha_1$. The relations involving $u$ can also be determined, but are not useful. We really need the corresponding lifted relations in $ER(2)$-cohomology, where they hold only mod$(x)$. Again, we have

\begin{equation}
u^{8K+6} = x^2z_{16K-14}, \,\, u^{8K+7}=x^4 z_{16K+4}, \,\, u^{8K+8}=0.
\end{equation}

Now in $ER(2)$-cohomology, $q^*i_{16K+10} = v_2u^{8K+5}$, from  \cite[(13.1)]{KW1}. Since the degrees of $\alpha_2, \alpha, w$ and $x$ are $-12s,16,40$ and $-17$ respectively, mod 48, the only elements in $ER(2)^*i_{16K+10}$ of degree $-16(8K+5)$ have the form $\alpha ^k w x^2 i_{16K+10}$. Then we have $q^*(\alpha ^q w x^2 i_{16K+10})=v_2\alpha ^q w x^2 u^{8K+5}$, which lies in the ideal $(x)$ and is not the same as $u^{8K+5}$.
\end{proof}

\subsection{The Bockstein spectral sequence for $ER(2)^*(RP^{16K+9})$}

We compute this cohomology by sandwiching it between the $ER(2)$-cohomology of $RP^{16K+10}$ and $RP^{16K+8}$, which we know, in the commutative diagram of exact sequences

\begin{equation}
\xymatrix{
ER(2)^*(S^{16K+10},*) \ar[d]^{\rho^*} \ar[r]^= &ER(2)^*(S^{16K+10},*) \ar[d]\\
ER(2)^*(\frac{RP^{16K+10}}{RP^{16K+8}},*) \ar[r] \ar[d]^{i^*} &ER(2)^*(RP^{16K+10},*)\ar[d] \ar[r] &ER(2)^*(RP^{16K+8},*) \ar[d]^=\\
ER(2)^*(S^{16K+9},*) \ar[r] \ar[d]^2 &ER(2)^*(RP^{16K+9},*) \ar[r] &ER(2)^*(RP^{16K+8},*)\\
ER(2)^*(S^{16K+9},*)\\
}
\end{equation}

The $E^1$-term for the Bockstein spectral sequence is just $E(2)^*(RP^{16K+9},*)$, which decomposes \cite{Davis84} as $$E(2)^*(RP^{16K+8},*)\oplus E(2)^*(S^{16K+9},*).$$ via the maps $$RP^{16K+8} \rightarrow RP^{16K+9} \rightarrow S^{16K+9}$$ 
(In general, $E(2)^*(RP^{2n})=E(2)^*[u]/(u^{n+1},[2](u))$, as $E(2)$ is complex oriented and $E(2)^*$ has no 2-torsion.) The even part of the $E^1$-term has the 2-adic basis $$ v_2^i\alpha ^k u^j \hspace{5mm} (0\leq i \le 7, \hspace{5mm} 0 \le k, \hspace{5mm} 0\le j \leq 8K+4) $$ 
The odd part is a free $\mathbb{Z}_{(2)}$-module with basis $$v_2^i\alpha^k i_{16K+9} \hspace{5mm} (0\leq i \le 7, \hspace{5mm} 0 \leq q, \hspace{5mm} 0\leq k)$$
Since $d^1$ is of even degree we can just read off our $d^1$ from \cite[Theorem 13.2]{KW1} for $RP^{16K+8}$ and section 5 of \cite{KW1} for the $S^{16K+9}$ part.

$$d^1(v_2^{2s-5}\alpha^ku^j) = 2v_2^{2s}\alpha^ku^j = v_2^{2s}\alpha^{k+1}u^{j+1} +\ldots \hspace{5mm} (j\le 8K+4)$$ $$d^1(v_2^{2s+1}\alpha^k i_{16K+9})=2v_2^{2s-2}\alpha^k i_{16K+9}$$

Thus $E^2$ is given by $$v_2^{2s}\alpha^k u \hspace{5mm} (k \geq 0), \hspace{5mm} v_2^{2s}u^j \hspace{5mm} (1\le j \leq 8K+4), \hspace{5mm}v_2^{2s+1}\alpha^k u^{8K+4} \hspace{5mm} (0\leq k),$$ $$v_2^{2s}\alpha^k i_{16K+9}$$ 

$d^2$ has odd degree 35. Since we have both odd and even degree elements in the $E^2$-term, $d^2$ might very well be non-trivial. If it is, then by naturality, it must have its source in the $RP^{16K+8}$ part and target in the $S^{16K+9}$ part. Also, the source cannot be anything from the $E^2$-term for the BSS for $RP^{\infty}$, for we know that $d^2$ is trivial there. Therefore the only possible sources are $v_2^{2s+1}\alpha^k u^{8K+4}$ with possible targets $v_2^{2s}\alpha^k i_{16K+9}.$

Now, since $v_2^2$ is a unit, if there is a $d^2$, it must be non-zero on $v_2^{-1}u^{8K+4}$ which has degree $+6-16(8K+4)$ which is $-10+16K$ mod 48. The degree of the target must be this plus 35. The possible targets have degrees $-12s-32k+16K+9$. The only solutions are $\alpha i_{16K+9}$, $\alpha^4 i_{16K+9}$ etc.

If $d^2$ is non-zero our guess would be $d^2(v_2^{-1}u^{8K+4})=\alpha i_{16K+9}$. Thus we need to show that $x^2\alpha i_{16K+9}=0$ in order for our guess to be correct.

The left column of (11) shows that the $ER(2)^*$-module $ER(2)^*(RP^{16K+10}/RP^{16K+8},*)$ is generated by two elements $z_{16K+10} = \rho^* i_{16K+10}$ and $z_{16K-8}$, where $i^*z_{16K-8}=xi^{16K+9}$.

We want to show that $x^2\alpha i_{16K+9}=0$ in $ER(2)^*(RP^{16K+9},*)$. This is the image of the same named element in $ER(2)^*(S^{16K+9},*)$ which lifts to $x\alpha z_{16K-8} \in ER(2)^*({RP^{16K+10} \over RP^{16K+8}})$. This maps to $x\alpha v_2^4u^{8K+5}$ in $ER(2)^*(RP^{16K+10},*)$ from \cite[13.1]{KW1}. Since $2x=0$, we can use the relation $[2](u) = 2u +_F \alpha u^2 +_F u^4 =0$ on $\alpha v_2^2$, and we get $$x\alpha v_2^4 u^{8K+5} =xv_2^4u^{8K+7} + \ldots$$ 
The least power of $u$ which is zero in $ER(2)^*(RP^{16K+10})$ is $8K+8$. We have noted that $u^{8K+7}=x^4z_{16K+4}$, so that we have $x^5v_2^4z_{16K+4}$, which is non-zero and does not lift to $S^{16K+10}$. It follows that $x^2\alpha i_{16K+9} \neq 0$.

However, if we multiply the whole calculation by $\alpha^3$, we get $\alpha^3 x^5 v_2^4 z_{16K+4}=0$, as $x^3\alpha=0$. So $x^2\alpha^4 i_{16K+9}=0$, and we conclude that $$d^2(v_2^{2s+1}\alpha^k u^{8K+4}) = v_2^{2s+2}\alpha ^{k+4} i_{16K+9}$$

Then, $E^3$ is given by $$v_2^{2s}\alpha ^k u \hspace{3mm} (0\leq k),\,\,\,v_2^{2s}u^j \hspace{3mm} (1\le j \leq 8K+4),\,\,\,v_2^{2s}\alpha^{\{0-3\}}i_{16K+9}.$$ 
 $d^3$ is even degree so the even and odd parts don't mix under the differential. On both parts we already know the $d^3$ differential: $$d^3(v_2^{\{6,2\}}\alpha^ku)=v_2^{\{0,4\}}\alpha^{k+1}u$$ $$d^3(v_2^{\{6,2\}}u^j)=v_2^{\{0,4\}}\alpha u^j =v_2^{\{0,4\}}u^{j+2} \hspace{5mm} (1\le j\leq 8K+2)$$
$$d^3(v_2^{4s-2}\alpha^{\{0-2\}}i_{16K+9})=v_2^{4s}\alpha^{\{1-3\}}i_{16K+9} $$ $$d^3(v_2^{4s}\alpha^{\{0-3\}}i_{16K+9}) =0$$
Thus $E^4$ is given by $$v_2^{\{0,4\}}u^{\{1-3\}},\hspace{5mm} v_2^{\{6,2\}}u^{{8K+3,8K+4}}, \hspace{5mm} v_2^{4s} i_{16K+9}. $$

$d^4$ has degree 21, which is odd. So it must go from the $RP^{16K+8}$ part to the $S^{16K+9}$ part. $d^4$ must be zero on anything in the image from $RP^{\infty}$. So our non-zero differentials have possible sources $v_2^{\{6,2\}}u^{\{8K+3,8K+4\}}$, and possible targets $v_2^{4s}i_{16K+9}$ and $v_2^{\{6,2\}}\alpha^3i_{16K+9}$.
Let's compute the degrees (mod 48).
$$v_2^6u^{8K+3}:-36-16(8K+3) \equiv -36+16K \longrightarrow -15+16K$$
$$v_2^2u^{8K+3}:-12-16(8K+3) \equiv -12+16K \longrightarrow 9+16K$$
$$v_2^6u^{8K+4}:-36-16(8K+4) \equiv -4+16K \longrightarrow 17+16K$$
$$v_2^2u^{8K+4}:-12-16(8K+4) \equiv 20+16K \longrightarrow 41+16K$$

Comparing with degrees of the possible targets we see that the only possible differentials are:
$$d^4(v_2^{\{6,2\}}u^{8K+3})=v_2^{\{0,4\}}i_{16K+9}$$
This must be true if $i_{16K+9}$ is $x^4$-torsion. We invoke the commutative diagram (11) used before. Consider $x^3z_{16K-8}$ in $ER(2)^*({RP^{16K+10} \over RP^{16K+8}})$. The following diagram shows its images in the lower left hand square of the commutative diagram.

$$\xymatrix{
x^3z_{16K-8} \ar@{|->}[d] \ar@{|->}[r] &x^3v_2^4u^{8K+5} \ar@{|->}[d]\\
x^4i_{16K+9} \ar@{|->}[r] &x^4i_{16K+9}\\
}$$
The element in the upper right-hand corner is zero since $x^3u^4=0$. This shows $i_{16K+9}$ is indeed $x^4$-torsion. We obtain our $E^5$-term

$$v_2^{\{0,4\}}u^{\{1-3\}},\hspace{3mm} v_2^{\{6,2\}}u^{8K+4}. $$ 
$d^5$ has even degree. For dimensional reasons the differentials must be zero in the odd part, and \cite[Theorem 13.2]{KW1} determines that it is zero for the even part.

\vspace{5mm}
$d^6$ has degree 7. Again, it must go from even part to odd part by naturality. By sparseness, $d^6 = 0$.

Since $d^7$ has even degree, the differential does not mix odd and even degrees. First of all in the even part we have $$d^7(v_2^4u^{\{1-3\}})=u^{\{1-3\}}$$ Also by mapping to $RP^{16K+8}$ \cite[page 23]{KW1} we get that $$d^7(v_2^6u^{8K+4})=v_2^2u^{8K+4}$$

We collect our results in the following theorem.

\vspace{7mm}
\begin{theorem} The Bockstein spectral sequence for $ER^*(RP^{16K+9},*)$ is as follows:

$E^1$ $$v_2^i\alpha^k u^j \hspace{5mm} (0\leq i \le 7, \hspace{5mm} 0\leq k, \hspace{5mm} 0\leq j \leq 8K+4)$$ $$v_2^i \alpha^k i_{16K+9} \hspace{5mm} (0\leq i \leq 7, \hspace{5mm} 0\leq k)$$
$$d^1(v_2^{2s-5}\alpha^ku^j)=2v_2^{2s}\alpha^ku^j = v_2^{2s}\alpha^{k+1}u^{j+1}+\ldots \hspace{5mm} (j\le 8K+3)$$ $$d^1(v_2^{2s+1}\alpha^ki_{16K+9})=2v_2^{2s-2}\alpha^ki_{16K+9}$$ where $v_2^i\alpha^ki_{16K+9}$ generates a free $\mathbb{Z}_{(2)}$-module in $E^1$.

\vspace{5mm}
$E^2$ 
 $$v_2^{2s}\alpha^k u \hspace{5mm} (k \geq 0), \hspace{5mm} v_2^{2s}u^j \hspace{5mm} (1\le j \leq 8K+4), \hspace{5mm}v_2^{2s+1}\alpha^k u^{8K+4} \hspace{5mm} (0 \leq k),$$ $$v_2^{2s}\alpha^k i_{16K+9}$$
 $$d^2(v_2^{2s+1}\alpha^k u^{8K+4}) = v_2^{2s+2}\alpha ^{k+4} i_{16K+9}$$

\vspace{5mm}
$E^3$  $$v_2^{2s}\alpha ^k u \hspace{5mm} (0\leq k),\,\,\,v_2^{2s}u^j \hspace{5mm} (1\le j \leq 8K+4),\,\,\,v_2^{2s}\alpha^{{0-3}}i_{16K+9}$$
 $$d^3(v_2^{\{6,2\}}\alpha^ku)=v_2^{\{0,4\}}\alpha^{k+1}u$$ $$d^3(v_2^{\{6,2\}}u^j)=v_2^{\{0,4\}}\alpha u^j =v_2^{\{0,4\}}u^{j+2} \hspace{5mm} (2\le j\leq 8K+2)$$
 $$d^3(v_2^{4s-2}\alpha^{\{0-2\}}i_{16K+9})=v_2^{4s}\alpha^{\{1-3\}}i_{16K+9} $$ $$d^3(v_2^{4s}\alpha^{\{0-3\}}i_{16K+9}) =0$$
 
\vspace{5mm}
$E^4$  $$v_2^{\{0,4\}}u^{\{1-3\}},\hspace{5mm} v_2^{\{6,2\}}u^{\{8K+3,8K+4\}},\hspace{5mm} v_2^{4s} i_{16K+9} $$
$$d^4(v_2^{\{6,2\}}u^{8K+3})=v_2^{\{0,4\}}i_{16K+9}$$

\vspace{5mm}
$E^5=E^6=E^7$ $$v_2^{\{0,4\}}u^{\{1-3\}},\hspace{3mm} v_2^{\{6,2\}}u^{8K+4},\hspace{3mm} v_2^{\{6,2\}}\alpha^3i_{16K+9}$$
$$d^7(v_2^4u^{\{1-3\}})=u^{\{1-3\}},\hspace{5mm} d^7(v_2^6u^{8K+4})=v_2^2u^{8K+4}$$

\end{theorem}
 
Next we identify all the elements in degree 8*.

\begin{theorem}
A 2-adic basis of $ER(2)^{8*}(RP^{16K+9},*)$ is given by the elements 
$$\alpha^ku^j, \,\, (k \geq 0,1\le j \le 8K+4)$$
$$v_2^4\alpha^k u^j, \,\, (k \geq 1, 1\leq j \leq 8K+4)$$
$$v_2^4u^j,\,\,(4 \leq j \leq 8K+4)$$
$$x \alpha^ki_{16K+9}, \,\, xv_2^4\alpha^k i_{16K+9},\,\,(k \geq 0)$$
\end{theorem}

\begin{proof} The first classes of elements represent the images of differentials in the spectral sequence that do not involve $i_{16K+9}$. As in \cite{KW2}, multiplication by powers of $x$ leads to no new elements in degree $8*$. Those images involving $i_{16K+9}$ provide $x^2$,$x^3$, or $x^3$-torsion, which may be multiplied by $x$. \end{proof}

\begin{cor}
There is an algebraic map $$ER(2)^{8*}(RP^{16K+9}) \rightarrow E(2)^{8*}(RP^{16K+10})$$ which only misses the elements $v_2^4u^{\{1-3\}}$.
\end{cor}

\section{Non-Immersions}

If $RP^b$ immerses in $\mathbb{R}^c$, James showed \cite{James} that there is an axial map
$$m: RP^b \times RP^{2^L-c-2} \rightarrow RP^{2^L-b-2}$$ for large $L$ (meaning a map that is non-trivial on both axes). Specifically, to show that $RP^{2n}$ does not immerse in $\mathbb{R}^{2k+1}$ we need to prove there is no axial map 
\begin{equation}
m: RP^{2n} \times RP^{2^L-2k-3} \rightarrow RP^{2^L-2n-2}
\end{equation}

Our strategy is to consider the class $u \in ER(2)^*(RP^{2^L-2n-2})$, which satisfies $u^{2^{L-1}-n}=0$ when $n\equiv$0 or 7 mod 8 \cite[Theorem 1.6]{KW1}. We shall see that $m^*u$ is known, in principle. If we can show that $(m^*u)^{2^{L-1}-n} \neq 0$, we have a contradiction.

Davis \cite{Davis84} used this approach, by using the complex-oriented cohomology theory $E(2)$ to deduce that $RP^{2n}$ does not immerse in $\mathbb{R}^{2k}$ by showing there is no axial map
$$m: RP^{2n} \times RP^{2^L-2k-2} \rightarrow RP^{2^L-2n-2}$$ when $n = m+ \alpha(m) -1$ and $k = 2m-\alpha(m)$ for some $m$, where $\alpha(m)$ denotes the number of 1's in the binary expansion of $m$. We wish to improve this result to show that for certain $n$ and $k$, (12) does not exist.

\vspace{3mm}
There is an axial map $m: RP^{\infty} \times RP^{\infty} \rightarrow RP^{\infty}$, which is the restriction of the map $\mathbb{C}P^{\infty} \times \mathbb{C}P^{\infty} \rightarrow \mathbb{C}P^{\infty}$ induced by the tensor product of the canonical Real line bundles. Therefore $m^*u = u_1 +_F u_2$, where $u_1,u_2$ and $u$ are the Chern classes for the three copies of $RP^{\infty}$. 

If $m:RP^b \times RP^c \rightarrow RP^d$ is an axial map, the diagram 
$$\xymatrix{
RP^b \times RP^c \ar[r]^m \ar[d]^{\subset} &RP^d \ar[d]^{\subset}\\
RP^{\infty} \times RP^{\infty} \ar[r]^m &RP^{\infty}
}$$
commutes, as all axial maps $RP^b \times RP^c \rightarrow RP^{\infty}$ are homotopic. It follows that the same formula $m^*u = u_1 +_F u_2$ holds for this $m$. As the formal group law $F$ is a formal power series in $u_1$ and $u_2$ over $\mathbb{Z}_{(2)}[\alpha]$ and $\deg u = -16$ and $\deg \alpha =16$, we are interested only in degrees that are multiples of 16. This simplifies our work, as $ER(2)^{16*}(RP^{\infty}) \rightarrow E(2)^{16*}(RP^{\infty})$ is an isomorphism by \cite{KW1}.

\vspace{3mm}
We assume $k = 2$ mod 8, so that  $2^L-2k-2 = 16K+10$ and we can use Theorem 3.2. Consider the diagram
$$\xymatrix{
ER(2)^{16*}(RP^{\infty} \times RP^{\infty}) \ar[r] \ar[d] &E(2)^{16*}(RP^{\infty} \times RP^{\infty}) \ar[d]\\
ER(2)^{16*}(RP^{2n} \times RP^{16K+10}) \ar[r] \ar[d] &E(2)^{16*}(RP^{2n} \times RP^{16K+10}) \ar[d]\\
ER(2)^{16*}(RP^{2n} \times RP^{16K+9}) \ar[r] &E(2)^{16*}(RP^{2n} \times RP^{16K+9})\\
}$$

From Don Davis's work, the image of $(u_1 +_F u_2)^{2^{L-1}-n} \in ER(2)^{16*}(RP^{\infty} \times RP^{\infty})$ in $E(2)^{16*}(RP^{2n} \times RP^{16K+10})$ is non-zero. We need to show that the image in $ER(2)^{16*}(RP^{2n} \times RP^{16K+9})$ is non-zero, (Note that we cannot use $E(2)$-cohomology for this purpose, as it is complex-oriented, which implies that $u^{8K+5} \in E(2)^*(RP^{16K+10})$ maps to zero in $E(2)^*(RP^{16K+9})$.

The two end terms in $(u_1 +_F u_2)^{2^{L-1}-n}$ are $u_1^{2^{L-1}-n}$ and $u_2^{2^{L-1}-n}$, which are plainly zero; all the other terms have the form $\lambda \alpha^ku_1^i u_2^j$, where $\lambda \in \mathbb{Z}_{(2)}, k \geq 0, i \geq 1, j \geq 1,$ and $i+j \geq 2^{L-1}-n$. Following \cite{KW1}, we may use the formulae $$2u_1=-\alpha u_1^2 + \ldots, \,\, \alpha u_1u_2^2 = \alpha u_1^2u_2 + \ldots $$ and induction to reduce $(u_1 +_F u_2)^{2^{L-1}-n}$ to a sum of distinct terms of the forms $\alpha^ku_1^iu_2$ and $u_1^iu_2^j$, with no numerical cofficient. (The first formula comes from $[2](u_1)=0$; the second from $u_1[2](u_2)- u_2[2](u_1)=0$.) Further, again by \cite{KW1}, $u_1^{n+1}=0$ since $n \equiv 0$ or 7 mod 8. Then $\alpha ^k u_1^i u_2=0$, since we still have $i+1 \geq 2^{L-1}-n.$ We do not know (or need to know) exactly which terms are present; all we have to do is show that the monomials $u_1^i u_2^j$ (for $1 \leq i \leq n$ and $1 \leq j \leq 8K+5$) remain linearly independent in $ER(2)^*(RP^{2n} \times RP^{16K+9})$, which we defer to the next section.  

\vspace{3mm}
Meanwhile, let us review the various numerical conditions. We need $n=m+\alpha(m)-1$, $k=2m-\alpha(m)$,$k\equiv 2$ mod 8 and $n \equiv$ 0 or 7 mod 8. So $2m -\alpha (m)\equiv 2$ and $m + \alpha(m) \equiv$ 0 or 1. Solving these, we get $(m,\alpha(m)) \equiv$ (6,2) or (1,0).

\vspace{5mm}
\begin{theorem}
If $(m,\alpha(m)) \equiv$ (6,2) or (1,0) mod 8, 

$RP^{2(m+\alpha(m)-1)}$ does not immerse in $\mathbb{R}^{2(2m -\alpha(m))+1}$.
\end{theorem}
 
\subsection{Products with an odd space}

As we discussed towards the end of the previous section, in order to complete the proof of theorem 4.1 we need to show that the monomials $u_1^iu_2^j$ (for $1 \leq i \leq n$ and $1 \leq j \leq 8K+5$) remain linearly independent in  $ER(2)^*(RP^{2n} \times RP^{16K+9})$. The argument presented in this section is similar to \cite[section 11]{KW2}.

We shall look into the Bockstein spectral sequence for $$ER(2)^*(RP^{2n} \wedge RP^{16K+9},*)$$ where $2n < 16K+9$.

The $E^1$-term is the usual $$E(2)^*(RP^{2n} \wedge RP^{16K+9},*)$$
$$\simeq E(2)^*(RP^{2n},*) \otimes E(2)^*(RP^{16K+9},*) \oplus \Sigma^{-16(8K+4)-1}E(2)^*(RP^{2n},*)$$
(from \cite{GW})
Also, we know that $$E(2)^*(RP^{16K+9},*) \cong E(2)^*(RP^{16K+8},*)\oplus E(2)^*(S^{16K+9},*).$$

Since $E(2)^*(S^{16K+9},*)$ is free it does not affect the Tor term, only the tensor product. So our $E^1$-term is:
$$E(2)^*(RP^{2n},*)\otimes E(2)^*(RP^{16K+8},*) \oplus E(2)^*(RP^{2n},*) \otimes E(2)^*(S^{16K+9},*)$$
$$\oplus \Sigma^{16K-17}E(2)^*(RP^{2n},*)$$
A 2-adic basis for this is given by 
$$v_2^s\alpha^k u_1^iu_2 \hspace{5mm} (0\leq k, \hspace{5mm} 0<i\le n, \hspace{5mm} s <8 )$$
$$v_2^su_1^iu_2^j \hspace{5mm} (0<i \le n, \hspace{5mm} 1 <j \le 8K+4, \hspace{5mm} s<8) $$
by the same reduction as before, and
$$v_2^s\alpha^k i_{16K+9} \hspace{5mm} (0\leq k, \hspace{5mm} 0 <i \le n, \hspace{5mm} s <8)$$
$$v_2^s\alpha^ku_1^iz_{16K-33} \hspace{5mm} (0\le k, \hspace{5mm} 0\le i < n, \hspace{5mm} s <8).$$

We know that $xi_{16K+9}$ represents $v_2^4u_2^{8K+4}$. So $v_2^4xu_1^ni_{16K+9}$ represents $u_1^nu_2^{8K+5}$. There is no differential on $u_1^ni_{16K+9}$. Also there is no differential on $v_2^4u_1^ni_{16K+9}$. All we have to do is show that $u_1^ni_{16K+9}$ is not in the image of $d^1$. Since $d^1$ has even degree we only have to worry about the odd degree elements since $u_1^ni_{16K+9}$ is odd degree.

$d^1$ has degree 18 so if it is to hit $u_1^ni_{16K+9}$ it must start at some $\alpha^k u_1^i z_{16K-33}$ because they are the only elements in the correct degree modulo 16. Then we would have $d^1$ non-trivial on $z_{16K-33}$.

In the Bockstein spectral sequence for $ER(2)^*(RP^{16M+16} \wedge  RP^{16K+10},*)$, $8M+8<8K+5$, we have from \cite[Theorem 19.2]{KW1}, $d^1(z_{16K-1})=0$. From \cite[Theorem 1.2]{GW} we have that $z_{16K-1}$ maps to $u_1z_{16K-33}$ in the spectral sequence for $ER(2)^*(RP^{16K+16} \wedge RP^{16K+10},*)$. Since this passes through the spectral sequence for $ER(2)^*(RP^{16M+16} \wedge RP^{16K+9},*)$, $z_{16K-1}$ maps to $u_1z_{16K-33}$ here as well. So $d^1(u_1z_{16K-33})= u_1d^1(z_{16K-33})=0$. 

$$\xymatrix{
ER(2)^*(RP^{16K+16} \wedge RP^{16K+10}) \ar[d] \ar[r] &ER(2)^*(RP^{16M+16} \wedge RP^{16K+8}) \\
ER(2)^*(RP^{16M+16} \wedge RP^{16K+9}) \ar[ur]\\
}$$

All elements killed by multiplication by $u_1$ go to zero under the map to $RP^{16K} \wedge RP^{16K+9}$, so our $d^1(z_{16K-33})$ is zero.

\begin{theorem}
When $n \le 8M < 8M+8 < 8K+5$, in $$ER(2)^{16*}(RP^{2n} \wedge RP^{16K+9})$$ the element $u_1^nu_2^{8K+5}$ is non-zero.
\end{theorem}

\vspace{5mm}

{\em Address:}\\
\texttt{A360 School of Mathematics,}\\
\texttt{Tata Institute of Fundamental Research,}\\
\texttt{1 Homi Bhabha Road,}\\ 
\texttt{Mumbai, India 400005}\\
\texttt{banerjee@math.tifr.res.in}

\end{document}